\documentclass[letterpaper]{article} 
\usepackage{amsmath}
\usepackage[retainorgcmds]{IEEEtrantools}
\usepackage{graphicx}
\usepackage[margin=1in]{geometry}
\usepackage{amsfonts}
\usepackage{amsthm}
\usepackage{setspace}
\usepackage{mathtools}
\usepackage{listings}
\usepackage{mathabx}
\usepackage{amssymb}
\usepackage[mathscr]{euscript}
\usepackage{pictexwd,dcpic}
\newtheorem{thm}{Theorem}
\theoremstyle{definition}
\newtheorem{lem}{Lemma}

\newtheorem{crl}{Corollary}

\newcommand{\Z}{\mathbb{Z}}

\newcommand{\R}{\mathbb{R}}

\newcommand{\C}{\mathbb{C}}

\author{Cole Hugelmeyer} 
\title{Every smooth Jordan curve has an inscribed rectangle with aspect ratio equal to $\sqrt{3}$}
\begin{document} 
\maketitle 
\abstract{We use Batson's lower bound on the nonorientable slice genus of $(2n,2n-1)$-torus knots to prove that for any $n \geq 2$, every smooth Jordan curve has an inscribed rectangle of of aspect ratio $\tan(\frac{\pi k}{2n})$ for some $k\in \{1,...,n-1\}$. Setting $n = 3$, we have that every smooth Jordan curve has an inscribed rectangle of aspect ratio $\sqrt{3}$.}

\section{Introduction}
Every Jordan curve has an inscribed rectangle. Vaughan proved this fact by parameterizing a M\"obius strip above the plane which bounds the curve, and for which self-intersections correspond to inscribed rectangles. \cite{survey} If we fix a positive real number $r$, we may ask if every Jordan curve has an inscribed rectangle of aspect ratio $r$. For all $r\neq 1$, this problem is open, even for smooth or polygonal Jordan curves. The case of $r = 1$ is the famous inscribed square problem, which has been resolved for a large class of curves. The known partial results include \cite{thesis}, in which it is shown that curves which are suitably close to being convex must have inscribed rectangles of aspect ratio $\sqrt{3}$, and \cite{convex}, in which it is shown that every convex curve has an inscribed rectangle of every aspect ratio. 

We present a 4-dimensional generalization of Vaughan's proof which lets us get some control on the aspect ratio of the inscribed rectangles. In particular, we will resolve the case of $r = \sqrt{3}$ for all smooth Jordan curves.

In regards to notation, $M$ will denote the M\"obius strip $\text{Sym}_2(S^1) = (S^1\times S^1)/(\Z/2\Z)$, where the $\Z/2\Z$ action consists of swapping the elements of the ordered pair. We write $\{x,y\}$ to denote the equivalence class represented by the pair $(x,y)$. As a notational convenience, we abide by the convention that $S^1$ denotes the unit complex numbers.

\section{The Proof}
\begin{thm}
Let $\gamma: S^1\to \C$ be a $C^\infty$ injective function with nowhere vanishing derivative.  Then for all integers $n\geq 2$, there exists an integer $k\in \{1,...,n-1\}$ so that $\gamma$ has an inscribed rectangle of aspect ratio $\tan(\frac{\pi k}{2n})$. 
\end{thm}

We first prove a lemma.

\begin{lem}
Let $K_n$ denote the knot in $\C\times S^1$ parameterized by $g\mapsto (g,g^{2n})$ for $g\in S^1$. Then if $n \geq 3$, there is no smooth embedding of the M\"obius strip $M\hookrightarrow \C\times S^1\times \R_{\geq 0}$ such that $\partial M$  maps to $K_n\times \{0\}$.
\end{lem}

\begin{proof}
The manifold $\C\times S^1$ is diffeomorphic to the interior of the solid torus, so given any embedding of the solid torus into $S^3$, the image of $K_n$ under this embedding yields a knot in $S^3$ which must have non-orientable 4-genus at most that that of $K_n$. By embedding the solid torus with a single axial twist, we can make the image of $K_n$ be the torus knot $T_{2n,2n-1}$. In \cite{Batson}, Batson uses Heegaard-Floer homology to show that the non-orientable 4-genus of $T_{2n,2n-1}$ is at least $n-1$. This means that for $n\geq 3$, the knot $K_n$ cannot be bounded by a M\"obius strip in the 4-manifold $\C\times S^1\times \R_{\geq 0}$.
\end{proof}
\begin{proof}[Proof of Theorem 1]
Suppose $\gamma: S^1\to \C$ is a $C^\infty$ injective function with nowhere vanishing derivative, and no inscribed rectangles of aspect ratio $\tan(\frac{\pi k}{2n})$ for any $k\in \{1,...,n-1\}$. By the smooth inscribed square theorem, we can assume $n\geq 3$. We define $\mu: M\to \C^2$ by the formula $$ \mu\{x,y\} = \left( \frac{\gamma(x) + \gamma(y)}{2},  (\gamma(y) - \gamma(x))^{2n} \right) $$ We see that if $\mu\{x,y\} = \mu\{w,z\}$, then the pairs share a midpoint $m$, and the angle at $m$ between $x$ and $w$ must be a multiple of $\pi/n$. Therefore, either $\{x,y\} = \{w,z\}$ or $(x,w,y,z)$ form the vertices of a rectangle of aspect ratio $\tan(\frac{\pi k}{2n})$ for some $k\in \{1,...,n-1\}$. Furthermore, we see that the differential of $\mu$ is non-degenerate. Therefore, $\mu$ is a smooth embedding of $M$ into $\C^2$. If we take a small tubular neighborhood $N$ around $\C\times\{0\}$, the M\"obius strip $\mu(M)$ intersects $\partial N$ at a knot which is isotopy equivalent to the knot $K_n$ described in our lemma. We have therefore constructed a smooth M\"obius strip bounding $K_n$, which contradicts Lemma 1. 
\end{proof}

\begin{crl}
Every smooth Jordan curve has an inscribed rectangle of aspect ratio $\sqrt{3}$. 
\end{crl}

\begin{proof}
Apply Theorem 1 to $n = 3$. We have a rectangle of aspect ratio $\tan(\frac{1}{6\pi}) = \frac{1}{\sqrt{3}}$ or $\tan(\frac{2}{6\pi}) = \sqrt{3}$. Aspect ratio is only defined up to reciprocals, so either way, we have a rectangle of the desired aspect ratio. 
\end{proof}

\nocite{*}

\bibliography{Refrences}{}

\begin{thebibliography}{1}

\bibitem{convex}
Arseniy Akopyan and Sergey Avvakumov.
\newblock Any cyclic quadrilateral can be inscribed in any closed convex smooth
  curve.
\newblock {\em arXiv:1712.10205v1 [math.MG]}, 2017.

\bibitem{Batson}
Joshua Batson.
\newblock Nonorientable four-ball genus can be arbitrarily large.
\newblock {\em Mathematical Research Letters}, 21(3), 2014.

\bibitem{pegs}
H.B. Griffiths.
\newblock The topology of square pegs in round holes.
\newblock {\em Proceedings of the London Mathematical Society},
  s3-62(3):647--672, 1991.

\bibitem{thesis}
Benjamin Matschke.
\newblock {\em Equivariant topology methods in discrete geometry.}
\newblock PhD thesis, Freie Universit ̈at Berlin, 2011.

\bibitem{survey}
Benjamin Matschke.
\newblock A survey on the square peg problem.
\newblock {\em Notices of the AMS}, 61(4), 2014.

\bibitem{symmetric}
M.~J. Nielsen and S.~E. Wright.
\newblock Rectangles inscribed in symmetric continua.
\newblock {\em Geometriae Dedicata}, 56(3):285--297, 1995.

\bibitem{tao}
Terence Tao.
\newblock An integration approach to the toeplitz square peg problem.
\newblock {\em Forum of Mathematics, Sigma}, (5), 2017.

\end{thebibliography}
\bibliographystyle{plain}

\end{document}